\documentclass[11pt]{article}
\usepackage[utf8]{inputenc} 
\usepackage{latexsym,amsmath,amssymb,textcomp}
\usepackage[all]{xy}
\usepackage[nottoc, notlof, notlot]{tocbibind}
\usepackage{geometry} 
\usepackage{graphicx} 
\usepackage{amsthm}
\usepackage{amsmath}
\usepackage{booktabs}
\usepackage{array} 
\usepackage{paralist}
\usepackage{verbatim} 
\usepackage{subfig} 
\usepackage{hyperref}
\usepackage{fancyhdr} 
\title{Trace maps for Mackey algebras.}
\author{Baptiste Rognerud}
\begin{document}
\maketitle
\theoremstyle{plain}
\newtheorem{theo}{Theorem}[section]
\newtheorem{theox}{Theorem}[section]
\renewcommand{\thetheox}{\Alph{theox}}
\newtheorem{prop}[theo]{Proposition}
\newtheorem{lemme}[theo]{Lemma}
\newtheorem{coro}[theo]{Corollary}
\newtheorem{question}[theo]{Question}
\newtheorem{notation}[theo]{Notations}
\theoremstyle{definition}
\newtheorem{de}[theo]{Definition}
\newtheorem{dex}[theox]{Definition}
\theoremstyle{remark}
\newtheorem{ex}[theo]{Example}
\newtheorem{re}[theo]{Remark}
\newtheorem{res}[theo]{Remarks}
\renewcommand{\labelitemi}{$\bullet$}
\newcommand{\comu}{co\mu}
\newcommand{\mo}{\ensuremath{\hbox{mod}}}
\newcommand{\Mo}{\ensuremath{\hbox{Mod}}}
\newcommand{\decale}[1]{\raisebox{-2ex}{$#1$}}
\newcommand{\decaleb}[2]{\raisebox{#1}{$#2$}}
\begin{abstract}Let $G$ be a finite group and $R$ be a commutative ring. The Mackey algebra $\mu_{R}(G)$ shares a lot of properties with the group algebra $RG$ however, there are some differences. For example, the group algebra is a symmetric algebra and this is not always the case for the Mackey algebra. In this paper we present a systematic approach to the question of the symmetry of the Mackey algebra, by producing symmetric associative bilinear forms for the Mackey algebra. 
\newline Using the fact that the category of Mackey functors is a closed symmetric monoidal category, we prove that the Mackey algebra $\mu_{R}(G)$ is a symmetric algebra if and only if the family of Burnside algebras $(RB(H))_{H\leqslant G}$ is a family of  symmetric algebras with a compatibility condition.
\newline As a corollary, we recover the well known fact that over a field of characteristic zero, the Mackey algebra is always symmetric. Over the ring of integers the Mackey algebra of $G$ is symmetric if and only if the order of $G$ is square free. Finally, if $(K,\mathcal{O},k)$ is a $p$-module system for $G$, we show that the Mackey algebras $\mu_{\mathcal{O}}(G)$ and $\mu_{k}(G)$ are symmetric if and only if the Sylow $p$-subgroups of $G$ are of order $1$ or $p$. 
\end{abstract}
\par\noindent
{\it{\footnotesize   Key words: Finite group. Mackey functor. Symmetric Algebra. Symmetric monoidal category. \newline Burnside Ring.}}
\par\noindent
{\it {\footnotesize A.M.S. subject classification: 19A22, 20C05, 18D10,16W99.}}
\section{Trace maps for Mackey algebras.}
\subsection{Introduction.}
Let $R$ be a unital commutative ring and $G$ be a finite group.
The notion of Mackey functor was introduced by Green in $1971$. For him a Mackey functor is an axiomatisation of the comportment of the representations of a finite group. There are now several possible definitions of Mackey functors, in this paper we use the point of view of Dress who defined the Mackey functors as particular bivariant functors and we use the Mackey algebra introduced by Th\'evenaz and Webb. In \cite{tw} they proved that a Mackey functor is nothing but a module over the so-called Mackey algebra. Numerous properties of this algebra are known: this algebra shares a lot of properties with the group algebra. For example, the Mackey algebra is a free $R$-module, and its $R$-rank doesn't depend on the ring $R$. If we work with a $p$-modular system which is ``large enough'', there is a decomposition theory, in particular the Cartan matrix of this algebra is symmetric. However there are some differences, over a field of characteristic $p>0$, where $p\mid |G|$, the determinant of the Cartan matrix is not a power of the prime number $p$ in general, and as shown in $\cite{tw}$ the Mackey algebra is seldom a self-injective algebra. One may wonder about a stronger property for the Mackey algebra: when is the Mackey algebra a symmetric algebra? The answer to this question depends on the ring $R$.  
\newline When $R$ is a field of characteristic $0$ or coprime to $|G|$, the Mackey algebra is semi-simple (see \cite{tw_simple}), so it is clearly a symmetric algebra. Over a field of characteristic $p>0$ which is ``\emph{large enough}", where $p\mid |G|$, then Jacques Th\'evenaz and Peter Webb proved that the so called $p$-local Mackey algebra (see \cite{bouc_resolution}) is self-injective if and only if the Sylow $p$-subgroups of $G$ are of order $p$. However, in the same article, they proved that the $p$-local Mackey algebra is a product of matrix algebras and Brauer tree algebras. Since a Brauer tree algebra is derived equivalent to a symmetric Nakayama algebra, then by \cite{rickard_derived} or, for a more general result \cite{zimmermann_tilted_orders}, all Brauer tree algebras are symmetric algebras. So the $p$-local Mackey algebra over a field of characteristic $p$ is symmetric if and only if the Sylow $p$-subgroups are of order $1$ or $p$. Now the Mackey algebra of the group $G$ is Morita equivalent to a direct product of $p$-local Mackey algebras for some sub-quotients of the group $G$ (Theorem 10.1 \cite{tw}), so if $p^2 \nmid |G|$, the Mackey algebra of $G$ is symmetric. However, if $(K,\mathcal{O},k)$ is a $p$-modular system for the group $G$, it is not so clear that the previous argument can be use for the valuation ring $\mathcal{O}$. In particular the Mackey algebras over the valuation rings are rather complicate objects (see Section $6.3$ of \cite{these}).
\newline An $R$-algebra is a symmetric algebra if it is a projective $R$-module and if there exist a non degenerate symmetric, associative bilinear form on this algebra. One may think that the previous argument for the symmetry of the Mackey algebra is somewhat elaborate for something as elementary as the existence of a bilinear form on this algebra. However, for the Mackey algebra it is not obvious to specify such a bilinear form even in the semi-simple case.
\newline In this paper we propose a systematic approach to this question: by using the so-called Burnside Trace, introduce by Serge Bouc (\cite{bouc_burnside_dim}), we reduce the question of the existence of such bilinear a form on the Mackey algebra to the question of the existence of a family of symmetric, associative, non degenerate bilinear forms on Burnside algebras with an extra property. Here we denote by $RB(H)$ the usual Burnside algebra of the group $H$. 
\begin{dex} Let $G$ be a finite group and $R$ be a commutative ring. Let  $\phi=(\phi_{H})_{H\leqslant G}$ be a family of linear maps such that $\phi_{H}$ is a linear form on $RB(H)$. Let $b_{\phi_{H}}$ be the bilinear form on $RB(H)$ defined by $b_{\phi_{H}}(X,Y):= \phi_{H}(XY)$ for $X,Y\in RB(H)$. 
\begin{enumerate}
\item The family $\phi$ is stable under induction if for every $H$ subgroup of $G$ and finite $H$-set $X$ we have $\phi_{G}(Ind_{H}^{G}(X)) = \phi_{H}(X)$.
\item The family $\big(RB(H)\big)_{H\leqslant G}$ is a stable by induction family of symmetric algebras if there exist a stable by induction family of linear forms $\phi=(\phi)_{H\leqslant G}$ such that the bilinear form $b_{\phi_{H}}$ on $RB(H)$ is non-degenerate for all $H\leqslant G$.
\end{enumerate}
\end{dex}
The main result of the paper is the following theorem:
\begin{theox}
Let $G$ be a finite group and $R$ be a commutative ring. Then the Mackey algebra $\mu_{R}(G)$ is a symmetric algebra if and only if $\big(RB(H)\big)_{H\leqslant G}$ is a stable by induction family of symmetric algebras.
\end{theox}
As corollary, we produce various symmetric associative bilinear form on the Mackey algebra which generalize the usual bilinear form for the group algebra. Using these forms we give direct and elementary proof for the symmetry of the Mackey algebras in the following cases: 
\begin{itemize}
\item Over the ring of the integers $\mathbb{Z}$, the Mackey algebra of a finite group $G$ is symmetric if and only if the order of $G$ is square-free.
\item Over a field $k$ of characteristic $0$, the Mackey algebra of $G$ is symmetric. 
\item Over a field $k$ of characteristic $p>0$, the Mackey algebra of $G$ is symmetric if and only if $p^2\nmid |G|$.
\item Let $p$ be a prime number such that $p\mid |G|$. Let $R$ be a ring in which all prime divisors of $|G|$, except $p$, are invertible. Then the Mackey algebra $\mu_{R}(G)$ is symmetric if and only if $p^2 \nmid |G|$. In particular, if $(K,\mathcal{O},k)$ is a $p$-modular system for $G$, then the Mackey algebras $\mu_{k}(G)$ and $\mu_{\mathcal{O}}(G)$ are symmetric if and only of $p^2 \nmid |G|$.
\end{itemize}
\begin{notation}
We use the following notations:
\begin{itemize}
\item Let $G$ be a finite group. Then $[s(G)]$ denotes a set of representatives of the conjugacy classes of subgroups of $G$.
\item Let $X$ be a finite $G$-set. We still denote by $X$ the isomorphism class of $X$ in the Burnside ring $B(G)$. 
\item All the $G$-sets are supposed to be finite.
\item Let $p$ be a prime number. Then $O^{p}(G)$ is the smallest normal subgroup of $G$ such that $G/O^{p}(G)$ is a $p$-group. A finite group $G$ is $p$-perfect if $O^{p}(G)=G$. 
\item Let $H$ and $K$ be two subgroups of $G$. We use the notation $H=_{G} K$ if $H$ and $K$ are conjugate in $G$.
\end{itemize}
\end{notation}
\subsection{Symmetric algebras}
Let $R$ be a commutative ring with unit. 
\begin{de}[Definition $2.3$ \cite{broue_higman}]
Let $A$ be an $R$-algebra. Then $A$ is a symmetric algebra if:
\begin{enumerate}
\item $A$ is a finitely generated projective $R$-module.
\item There exist a non degenerate, associative, symmetric bilinear form $b$ on $A$. That is a bilinear form $b$ such that:
\begin{itemize}
\item for $x$, $y$, $z\in A$ we have $b(xy,z)=b(x,yz)$.
\item For $x$ and $y$ in $A$, we have $b(x,y)=b(y,x)$.
\item The map from $A$ to $Hom_{R}(A,R)$ defined by $x\mapsto b(x,-)$ is an isomorphism of $R$-modules. 
\end{itemize}
\end{enumerate}
\end{de}
\begin{re}
Let $A$ be an $R$-algebra which is a finitely generated projective $R$-module. Then $A$ is a symmetric algebra if and only if $A$ is isomorphic to $Hom_{R}(A,R)$ as $A$-$A$-bimodule.
\end{re}
We have the following elementary result:
\begin{lemme}
Let $A$ be an $R$-algebra which is free of finite rank over $R$. Let $b$ be a bilinear form on $A$. Let $e:=(e_{1},\cdots,e_{n})$ be an $R$-basis of $A$.  Then $b$ is non-degenerate if and only if the matrix of $b$ in the basis $e$ is invertible. 
\end{lemme}
\begin{proof}
Lemma $2.2$ \cite{symmetric_bilinear_forms}.
\end{proof}
\subsection{Mackey functors.}
There are several possible definitions for the notion of Mackey functor for $G$ over $R$. In this paper we use two of them.  The first definition is due to Dress in \cite{dress}. 
\begin{de}\label{Dress}
A bivariant functor $M~=~(M^{*},M_{*})$ from $G$-$\rm{set}$ to $R$-$\rm{Mod}$ is a pair of functors from $G$-$\rm{set}\to$ $R$-$\rm{Mod}$ such that $M^{*}$ is a contravariant functor, and $M_{*}$ is a covariant functor. If $X$ is a $G$-set, then the image by the covariant and by the contravariant part coincide. We denote by $M(X)$ this image. A Mackey functor for $G$ over $R$ is a bivariant functor from $G$-set to $R$-Mod such that: 
\begin{itemize}
\item Let $X$ and $Y$ be two finite $G$-sets, $i_{X}$ and $i_{Y}$ the canonical injection of  $X$ (resp. $Y$) in $X\sqcup Y$, then $(M^{*}(i_X),M^{*}(i_Y))$ and $(M_{*}(i_X),M_{*}(i_Y))$ are inverse isomorphisms. 
\begin{equation*}
M(X)\oplus M(Y)\cong M(X\sqcup Y). 
\end{equation*}
\item If
\begin{equation*}
\xymatrix{
X\ar[r]^{a}\ar[d]^{b}& Y\ar[d]^{c} \\
Z\ar[r]^{d} & T 
}
\end{equation*}
is a pullback diagram of $G$-sets, then the diagram 
\begin{equation*}
\xymatrix{
M(X)\ar[d]_{M_{*}(b)} & M(Y)\ar[l]_{M^{*}(a)}\ar[d]^{M_{*}(c)}\\
M(Z) & M(T)\ar[l]_{M^{*}(d)}
}
\end{equation*}
is commutative. 
\end{itemize}
\end{de}
A morphism between two Mackey functors is a natural transformation of bivariant functors. Let us denote by $Mack_{R}(G)$ the category of Mackey functors for $G$ over $R$. Let us first recall an important example of Mackey functor: 
\begin{ex}\cite{bouc_green}\label{burnside}
If $X$ is a finite $G$-set, then the category of $G$-sets over $X$ is the category with objects $(Y,\phi)$ where $Y$ is a finite $G$-set and $\phi$ is a morphism from $Y$ to $X$. A morphism $f$ from $(Y,\phi)$ to $(Z,\psi)$ is a morphism of $G$-sets $f:Y\to Z$ such that $\psi\circ f=\phi$. 
\newline\indent The Burnside functor at $X$ is the Grothendieck group  of the category of $G$-sets over $X$, for relations given by disjoint union. This is a Mackey functor for $G$ over $R$ by extending scalars from $\mathbb{Z}$ to $R$. We denote by $RB$ the functor after scalar extension. 
\newline If $X$ is a $G$-set, the Burnside module $RB(X^2)$ has an $R$-algebra structure. The product of (the isomorphism classes of) $(X\overset{\alpha}{\leftarrow} Y \overset{\beta}{\rightarrow} X)$ and $(X\overset{\gamma}{\leftarrow} Z \overset{\delta}{\rightarrow} X)$ is given by (the isomorphism class of ) the pullback along $\beta$ and $\gamma$. 
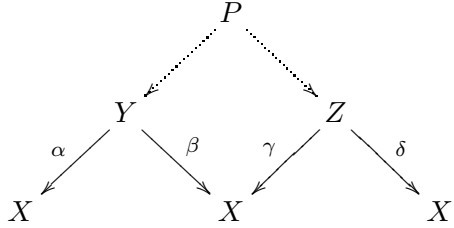
\begin{equation*}
\xymatrix{
& & P\ar@{..>}[dr]\ar@{..>}[dl] & & \\
& Y\ar[dl]_{\alpha}\ar[dr]^{\beta} & & Z\ar[dl]_{\gamma}\ar[dr]^{\delta} & \\
X & & X & &X
}
\end{equation*}
The identity of this $R$-algebra is (the isomorphism class of )$
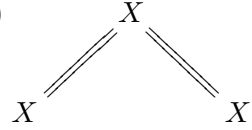
\xymatrix{&X\ar@{=}[rd]\ar@{=}[dl]&\\X& &X }$
\end{ex}
\begin{re}
The usual Burnside algebra of a finite group $H$, previously denoted by $RB(H)$ is isomorphic to the Burnside functor evaluated at the $H$-set $H/H$. In the Mackey functors' langage the first notation correspond to Green's notation and the second one correspond to Dress' notation. In the rest of the paper the notation $RB(H)$ will always be used for the usual Burnside algebra of the group $H$. If we want to speak about the Burnside functor evaluated at the $H$-set $H/1$, we will write $RB(H/1)$. 
\end{re}
Another definition of Mackey functors was given by Th\'evenaz and Webb in \cite{tw}. 
\begin{de}
The Mackey  algebra $\mu_{R}(G)$ for $G$ over $R$ is the unital associative algebra with generators  $t_{H}^{K}$, $r^{K}_{H}$ and $c_{g,H}$ for $H\leqslant K\leqslant G$ and $g\in G$, with the following relations:
\begin{itemize}
\item $\sum_{H\leqslant G}t^{H}_{H}=1_{\mu_{R}(G)}$. 
\item $t^{H}_{H}=r^{H}_{H}=c_{h,H}$ for $H\leqslant G$ and $h\in H$. 
\item $t^{L}_{K}t_{H}^{K}=t^{L}_{H}$, $r^{K}_{H}r^{L}_{K}=r^{L}_{H}$ for $H\subseteq K\subseteq L$. 
\item $c_{g',{^{g}H}}c_{g,H}=c_{g'g,H}$, for $H\leqslant G$ and $g,g'\in G$. 
\item $t^{{^{g}K}}_{{^{g}H}}c_{g,H}=c_{g,K}t^{K}_{H}$ and $r^{{^{g}K}}_{{^{g}H}}c_{g,K}=c_{g,H}r^{K}_{H}$, $H\leqslant K$, $g\in G$. 
\item $r^{H}_{L}t^{H}_{K}=\sum_{h\in [L\backslash H / K]} t^{L}_{L\cap {^{h} K}} c_{h, L^{h} \cap H} r^{K}_{L^{h}\cap H}$ for $L\leqslant H \geqslant K$. 
\item All the other products of generators are zero. 
\end{itemize}
\end{de}
\begin{prop}\label{basis}
The Mackey algebra is a free $R$-module, of finite rank independent of $R$. The set of elements $t^{H}_{K}xr^{L}_{K^{x}}$, where $H$ and $L$ are subgroups of $G$, where $x\in [H\backslash G/L]$, and $K$ is a subgroup of $H{\cap~{\ ^{x}L}}$ up to $(H\cap {\ ^{x}L})$-conjugacy, is an $R$-basis of $\mu_{R}(G)$.  
\end{prop}
\begin{proof}
Section $3$ of \cite{tw}. 
\end{proof}
\begin{prop}\label{prop_b}
The Mackey algebra $\mu_{R}(G)$ is isomorphic to $RB(\Omega_{G}^{2})$, where $\Omega_{G}$ is the $G$-set: $\sqcup_{L\leqslant G} G/L$. 
\end{prop}
\begin{proof}
The proof can be found in Proposition $4.5.1$ of \cite{bouc_green}. Let us recall that an explicit isomorphism $\beta$ can be defined on the generators of $\mu_{R}(G)$ by $\beta(t^{K}_{H}):=$
\begin{equation*}
\xymatrix{
& G/H\ar[dl]_{\pi_{H}^{K}}\ar@{=}[dr]& \\
G/K & & G/H
}
\end{equation*}
where $\pi_{H}^{K} : G/H\to G/K$ is the canonical map. \newline Similarly, we define $\beta(r^{K}_{H}):=$
\begin{equation*}
\xymatrix{
& G/H\ar[dr]^{\pi_{H}^{K}}\ar@{=}[dl]& \\
G/H & & G/K
}
\end{equation*}
and $\beta(c_{g,H}):=$
\begin{equation*}
\xymatrix{
& G/{^{g}H}\ar[dr]^{\gamma_{g,H}}\ar@{=}[dl]& \\
G/{^{g}H} & & G/H
}
\end{equation*}
where $\gamma_{g,H}(x{\ ^{g}H})=xgH$. One can check that this gives an isomorphism of algebras. 
\end{proof}
\begin{prop}[\cite{tw}]
There is an equivalence of categories \begin{center}$Mack_{R}(G)\cong \mu_{R}(G)$-Mod.\end{center}
\end{prop}
\subsection{Burnside Trace.}
There is a tensor product in the category of Mackey functors (see \cite{bouc_green}, e.g.). With this tensor product, the category is a closed symmetric monoidal category with the Burnside functor as unit. So, using the formalism of May (\cite{may_trace}) where the dualizable Mackey functors are exactly the finitely generated projective Mackey functors, Bouc has defined the notion of Burnside dimension and Burnside trace for these Mackey functors (\cite{bouc_burnside_dim}). Let $M$ be a finitely generated projective Mackey functor. The Burnside trace, denoted by $Btr$ is a map from $End_{Mack_{R}(G)}(M)$ to $RB(G)$. Let $RB_{X}$ be the Dress construction of the Burnside functor at the finite $G$-set $X$ (see \cite{dress} or \cite{bouc_green}). It is well known that $RB_{X}$ is a finitely generated projective Mackey functor. By an adjunction property, we have an isomorphism of $R$-algebras $End_{Mack_{R}(G)}(RB_{X})\cong RB(X^2)$ where the product on this ring is defined as in Example \ref{burnside}. Using these identifications, the Burnside trace on this Mackey functor is in fact a map from $RB(X^2)$ to $RB(G)$. Here we use Green's notation for $RB(G)$. 
\begin{prop}
Let $X$ and $Z$ be finite $G$-sets, let $a$ and $b$ be maps of $G$-sets from $Z$ to $X$. Let 
\begin{equation*}
f=\xymatrix{
& Z \ar[dl]_{b}\ar[dr]^{a}&\\
X && X
}
\end{equation*}
The Burnside trace $Btr : RB(X^2)\to RB(G)$ is defined on $f$ by: $$Btr(f):=\{z\in Z\ | a(z)=b(z)\}\in RB(G).$$ 
\end{prop}
\begin{proof}
Corollary 2.7 \cite{bouc_burnside_dim}. 
\end{proof}
By composing the Burnside trace by any $R$-linear map $RB(G)\to R$ we have a linear form on $RB(X^2)$. \\
\begin{re}\label{gene}
Let $R$ be a commutative ring. Let $f$ be a linear map from $RB(G)\to R$, such that $f(G/1)=1$. The trace map $f\circ Btr$ generalizes the usual trace map for the group ring $RG$ in the following way. The Burnside algebra $RB(G/1\times G/1)$ is isomorphic to $RG$. The isomorphism is defined as follow: a transitive $G$-set over $G/1\times G/1$ is isomorphic to
\begin{equation*}
f_{g} = \xymatrix{
& G/1\ar[rd]^{g}\ar@{=}[ld] &,\\
G/1 && G/1
}
\end{equation*}
for some $g\in G$.
The element $f_{g}$ is sent to $g\in RG$. Now, the Burnside trace of the element $f_{g}$ is $\delta_{g,1}G/1$. 
\end{re}
Using the fact that the Mackey algebra $\mu_{R}(G)$ is isomorphic to $RB(\Omega_{G}^{2})$, the Burnside trace gives a linear map from $\mu_{R}(G)$ to $RB(G)$. Using Proposition \ref{prop_b} we have as immediate corollary:
\begin{coro}
The Burnside Trace $Btr$ on the Mackey algebra is defined on a basis element~by \begin{equation*}
Btr(t^{K}_{H}xr^{L}_{H^{x}}) =\left\{\begin{array}{c}G/H \hbox{ if $K=L$ and $x\in L$} \\0 \hbox{ if not.}\end{array}\right.
\end{equation*}
\end{coro}
\begin{lemme}\label{calc}
Let $t^{H}_{K}xr^{L}_{K^{x}}$ and $t^{L}_{Q}yr^{H}_{Q^{y}}$ be two basis elements of $\mu_{R}(G)$. Then
\begin{equation*}
Btr\big(t^{H}_{K}xr^{L}_{K^{x}}t^{L}_{Q}yr^{H}_{Q^{y}}\big)=\sum_{\alpha\in[K^{x}\backslash L /Q]} \delta_{x\alpha y, H} G/(K\cap \ ^{x\alpha}Q),
\end{equation*}
where $\delta_{x\alpha y, H} =1$ if $x\alpha y\in H$ and $0$ otherwise. 
\end{lemme}
\begin{proof}
This follows from the computation of the product $t^{H}_{K}xr^{L}_{K^{x}}t^{L}_{Q}yr^{H}_{Q^{y}}$ by using the Mackey formula:
\begin{equation*}
Btr(t^{H}_{K}xr^{L}_{K^{x}}t^{L}_{Q}yr^{H}_{Q^{y}})=\sum_{\alpha\in[K^{x}\backslash L / Q] } Btr(t^{H}_{K\cap^{x\alpha}Q}\ x\alpha y\ r^{H}_{Q^{y}\cap K^{x\alpha y}}).
\end{equation*}
\end{proof}

Let $(-,-)_{B}$ be the bilinear map $\mu_{R}(G)\times \mu_{R}(G) \to RB(G)$ defined by \begin{center}$(x,y)_{B}:=Btr(xy)$ for $x,y\in \mu_{R}(G)$. \end{center}
\begin{lemme}\label{bl}
In the basis of Proposition \ref{basis} the matrix $M$ of the bilinear form $(-,-)_{B}$ is a permutation by block matrix. The possibly non-zero blocks can be labelled by $(H,L,x,y)$ where $H$ and $L$ are subgroups of $G$. The element $x$ is a representative of a double coset $H\backslash G/L$ and $y$ is a representative of $L\backslash G/H$ such that $HxL=Hy^{-1}L$.
\end{lemme}
\begin{proof}
In the basis of Proposition \ref{basis}, it is easy to see that the matrix $M$ of $(-,-)_{B}$ is a block matrix, where the blocks are indexed by two pairs of subgroups of $G$. Indeed the block matrix indexed by $(H,L)$ and $(M,N)$ is the sub-matrix of $M$ where the columns are indexed by the basis elements of the form $t^{H}_{K}xr^{L}_{K^{x}}$ and the lines are indexed by the basis elements of the form $t^{M}_{P}yr^{N}_{P^{y}}$. Now the product $t^{H}_{K}xr^{L}_{K^{x}}t^{M}_{P}yr^{N}_{P^{y}}$ is zero unless $L=M$ and $Btr\big(t^{H}_{K}xr^{L}_{K^{x}}t^{M}_{P}yr^{N}_{P^{y}}\big)=0$ unless $H=N$. So the non-zero blocks are exactly the blocks indexed by  the pairs of subgroups $(H,L)$ and $(L,H)$.
\newline Let $Bl$ be the block of $M$ indexed by $(H,L)$ and $(L,H)$. Then, the matrix $Bl$ is again a block matrix where the blocks are indexed by elements $x\in [H\backslash G / L]$ and $y\in [L\backslash G /H]$. Let us denote by $Bl_{H,L,x,y}$ the corresponding block.
\newline If $HxL \neq Hy^{-1} L$ then $Bl_{H,L,x,y} = 0$. Indeed if the restriction of $(-,-)_{B}$ to the block $Bl_{x,y}$ is non zero, then there are subgroups $K \leqslant  H\cap {\ ^{x}L}$ and $Q\leqslant L\cap {\ ^{y}H}$ and an element $\alpha\in [K^{x}\backslash L / Q]$ such that $x\alpha y \in H$. Then there exist $h\in H$ such that $x = hy^{-1}\alpha^{-1}$, so $HxL =Hy^{-1}L$. 
\end{proof}
\begin{notation} Let $\phi_{G}$ be a linear map from $RB(G)$ to $R$. 
\begin{itemize}
\item We denote by $tr_{\phi_{G}}$ the composite $\phi_{G}\circ Btr : \mu_{R}(G)\to R$.
\item We denote by $(-,-)_{\phi_{G}}$ the bilinear form on $\mu_{R}(G)$ defined by $(x,y)_{\phi_{G}}=tr_{\phi_{G}}(xy),$ for $x,y\in\mu_{R}(G)$. 
\item We denote by $b_{\phi_{G}}$ the bilinear form on $RB(G)$ defined by $b_{\phi_{G}}(X,Y):=\phi_{G}(XY)$. 
\end{itemize}
\end{notation}
\begin{lemme}
The map $tr_{\phi}$ is a central linear form on the Mackey algebra $\mu_{R}(G)$.
\end{lemme}
\begin{proof}
This follows from the fact that the Burnside trace is central.
\end{proof}
\begin{de} Let $G$ be a finite group and  $\phi=(\phi_{H})_{H\leqslant G}$ be a family of linear maps such that $\phi_{H}$ is a linear form on $RB(H)$. The family $\phi$ is stable under induction if for every $H$ subgroup of $G$ and finite $H$-set $X$ we have $\phi_{G}(Ind_{H}^{G}(X)) = \phi_{H}(X)$.
\end{de}
\begin{lemme}\label{red1}
Let $\phi=(\phi_{H})_{H\leqslant G}$ be a stable by induction family of linear forms on $\big(RB(H)\big)_{H\leqslant G}$. In the usual basis of $\mu_{R}(G)$, the matrix of $(-,-)_{\phi_{G}}$ is a permutation by block matrix. A non-zero block indexed by $(H,L,x,y)$ of this matrix is equal, up to permutation of the lines and the columns, to the block $(\Theta,\Theta,1,1)$ of the matrix of $(-,-)_{\phi_{\Theta}}$ for $\Theta = L\cap H^{x}$.
\end{lemme}
\begin{proof}
Let $Bl_{H,L,x,y}$ be a non-zero block of the matrix of $tr_{\phi_{G}}$. That is $H$ and $L$ are subgroups of $G$, the element $x$ is a representative of the double coset $H\backslash G/L$ and the element $y$ is a representative of $L\backslash G/H$. Since the block is non-zero, the double cosets $HxL$ and $Hy^{-1}L$ are equal. Let $h\in H$ and $l\in L$ such that $$y=lx^{-1}h.$$
Now the basis elements which appear for this block are: for the lines $t^{H}_{K}xR^{L}_{K^{x}}$ for $K\leqslant H\cap\ ^{x}L$ up to conjugacy in $H\cap\ ^{x}L$, and for the columns $t^{L}_{Q}yR^{H}_{Q^{y}}$ where $Q\leqslant L\cap\ ^{y} H$ up to conjugacy in $ L\cap\ ^{y} H$. By Lemma \ref{calc}, the entry indexed by this two elements is: $$\sum_{\alpha\in[K^{x}\backslash L /Q]} \delta_{x\alpha y, H} tr_{\phi_{G}}\big(G/(K\cap \ ^{x\alpha}Q)\big).$$
\begin{lemme}\label{le1} 
 The map $f$ defined by $f(\alpha)=\alpha l$ induces a bijection between the set $$\{ \alpha \in [K^{x}\backslash L /Q]\ ; x\alpha y \in H\},$$ and the set $$\{w\in [K^{x}\backslash L\cap H^{x}/ Q^{l}]\}.$$\end{lemme}
\begin{proof}
\begin{itemize}
\item Let $\alpha\in L$ such that $x\alpha y \in H$. Since $y=lx^{-1}h$ we have:
\begin{align*}
x\alpha y \in H &\Leftrightarrow x\alpha l x^{-1} h \in H\\
&\Leftrightarrow \alpha l \in H^{x},
\end{align*}
so $\alpha l \in L\cap H^{x}$.
\item The map $f$ is well defined: if $\alpha$ and $\alpha'$ are in the same double coset, there are $k\in J$ and $q\in Q$ such that $\alpha'=x^{-1}kx\alpha q$, and 
\begin{equation*}
f(\alpha')=x^{-1}kx\alpha q l = x^{-1}kx\alpha l l^{-1}q l,
\end{equation*}
so $f(\alpha)$ and $f(\alpha')$ are in the same double coset. 
\item The map $f$ is injective: if $f(\alpha)=f(\alpha')$ then there are $k\in K$ and $q\in Q$ such that $\alpha l = x^{-1}kx \alpha' l l^{-1}q l = x^{-1}kx \alpha' q l$, so $\alpha$ and $\alpha'$ are in the same double coset.
\item The map $f$ is surjective: let $w\in L\cap H^{x}$, then $wl^{-1}\in L$ and $f(wl^{-1}) = w$. 
\end{itemize}
\end{proof}
So, we have:
\begin{align*}
tr_{\phi_{G}}(t^{H}_{K}xR^{L}_{K^{x}}t^{L}_{Q}yR^{H}_{Q^{y}})&=\sum_{\alpha\in[K^{x}\backslash L /Q]} \delta_{x\alpha y, H} tr_{\phi_{G}}\big(G/(K\cap \ ^{x\alpha}Q)\big)\\
&=\sum_{w\in [K^{x}\backslash L\cap H^{x}/Q^{l}]} \phi_{G}(G/K\cap\ ^{xw}(Q^{l}))\\
&=\sum_{w\in [K^{x}\backslash L\cap H^{x}/Q^{l}]} \phi_{G}(G/K^{x}\cap\ ^{w}(Q^{l}))\\
&=\sum_{w\in [K^{x}\backslash L\cap H^{x}/Q^{l}]} \phi_{G}(Ind_{L\cap H^{x}}^{G}(L\cap H^{x}/K^{x}\cap\ ^{w}(Q^{l})))\\
&=\sum_{w\in [K^{x}\backslash L\cap H^{x}/Q^{l}]}\phi_{L\cap H^{x}}(L\cap H^{x}/K^{x}\cap\ ^{w}(Q^{l})). 
\end{align*}
Let $\Theta= L\cap\ ^{x}H$. The basis elements which appear for the block $Bl_{\Theta,\Theta,1,1}$ of the matrix of $\phi_{\Theta}$ are the $t^{\Theta}_{A}r^{\Theta}_{A}$  for $A\leqslant \Theta$ up to conjugacy. Let $A$ and $B$ be subgroups of $\Theta$, the entry corresponding to $t^{\Theta}_{A}r^{\Theta}_{A}$ and $t^{\Theta}_{B}r^{\Theta}_{B}$ is: 
$$ \sum_{w\in [A\backslash \Theta / B]} \phi_{\Theta}(\Theta/ A\cap\ ^{w} B).$$ So the blocks $B_{H,L,x,y}$ and $B_{\Theta,\Theta,1,1}$ are equals up to permutation of the lines and the columns. In particular, these two matrices have the same determinant, up to a sign. 
\end{proof}
\begin{lemme}\label{red2}
Let $\Theta$ be a finite group, and $\mu'$ the sub-algebra of $\mu_{R}(\Theta)$ generated by the elements of the form $t^{\Theta}_{A}r^{\Theta}_{A}$ for $A\leqslant \Theta$. Then the restriction of the Burnside trace to $\mu'$ is an isomorphism of $R$-algebras between $\mu'$ and $RB(\Theta)$, sending the basis of Proposition \ref{basis} to the usual basis of $RB(\Theta)$ consisting of isomorphism classes of transitive $G$-sets. 
\end{lemme}
\begin{proof}
It is clear that the restriction of the Burnside trace to $\mu'$ is an $R$-linear isomorphism since we have $Btr(t^{\Theta}_{A}r^{\Theta}_{A})= \Theta/A \in RB(\Theta)$. Moreover this is an isomorphism of algebras, since:
\begin{align*}
Btr(t^{\Theta}_{A}r^{\Theta}_{A}t^{\Theta}_{B}r^{\Theta}_{B})&=\sum_{\theta\in [A\backslash \Theta / B]} \Theta/(A\cap B^{\theta})\\
&=\Theta/A\times \Theta/B\in RB(\Theta). 
\end{align*}
\end{proof}
We have:
\begin{theo}\label{meta}
Let $G$ be a finite group. Let $\phi=(\phi_{H})_{H\leqslant G}$ be a stable by induction family of linear forms on $\big(RB(H)\big)_{H\leqslant G}$.  Then the bilinear form $(-,-)_{\phi_{G}}$ on the Mackey algebra $\mu_{R}(G)$ is non degenerate if and only if the bilinear form $b_{\phi_{H}}$ on $RB(H)$ is non degenerate for every $H$ subgroup of $G$. 
\end{theo}
\begin{proof}
If $\phi$ is such a family of linear forms, by Lemma \ref{bl} the matrix of the bilinear form $(-,-)_{\phi_{G}}$ in the usual basis of $\mu_{R}(G)$ is a permutation by block matrix. So the determinant of this matrix is (up to a sign) the product of the determinant of the non-zero blocks. By Lemma \ref{red1} and Lemma \ref{red2} the determinant of the block indexed by $(H,L,x,y)$ is equal to the determinant of the matrix of the bilinear form $b_{\phi_{L\cap H^{x}}}$ in the usual basis of $RB(L\cap H^{x})$. So the determinant of $(-,-)_{\phi_{G}}$ is invertible in $R$ if and only if the determinant of the form $b_{\phi_H}$ on $RB(H)$ is invertible in $R$ for every subgroup $H$ of $G$. 
\end{proof}
\begin{de}
Let $G$ be a finite group. The family $\big(RB(H)\big)_{H\leqslant G}$ is a stable by induction family of symmetric algebras if there exist a stable by induction family of linear forms $\phi=(\phi)_{H\leqslant G}$ such that the bilinear form $b_{\phi_{H}}$ on $RB(H)$ is non-degenerate for every $H\leqslant G$. 
\end{de}
\begin{theo}\label{main}
Let $G$ be a finite group. Then the Mackey algebra is a symmetric algebra if and only if $\big(RB(H)\big)_{H\leqslant G}$ is a stable by induction family of symmetric algebras.
\end{theo}
\begin{proof}
 Only for this proof, we use Green's definition of Mackey functors since it is much more convenient for understanding the action of the induction and restriction maps (see Section $2$ of \cite{tw}).
If $\big(RB(H)\big)_{H\leqslant G}$ is a stable by induction family of symmetric algebras, then by Theorem \ref{meta}, the Mackey algebra is symmetric. Conversely, if the Mackey algebra is symmetric, then the Mackey algebra is isomorphic to its $R$-linear dual as bimodule. Using the usual equivalence of categories, the modules over the Mackey algebras  are the Mackey functors. In particular the Burnside functor $RB$ corresponds to a direct summand of the free module of rank $1$ over the Mackey algebra. Since the Mackey algebra is symmetric, the Burnside functor is isomorphic to its $R$-linear dual, that is there there exist an isomorphism of Mackey functors $f: RB\to Hom_{R}(RB,R)$. For the Mackey functor structure of $Hom_{R}(RB,R)$, see Section $4$ of \cite{tw}. This isomorphism allows us to build an associative non-degenerate bilinear form ${<}-,-{>} : RB\times RB\to R$ i-e a family of bilinear form ${<}-,-{>}_{K}$ for each subgroup $K$ of $G$ defined in the following way: let $K$ be a subgroup of $G$ and $X$ and $Y$ be two elements of $RB(K)$, then $${<}X,Y{>}_{K}:=f_{K}(X)(Y)$$The fact that $f$ is a Mackey functor morphism implies in particular the following properties: let $H\leqslant K$ be subgroups of $G$, then: let $X$ be an $H$-set and $Y$ be an $K$-set, then: 
\begin{equation*}
{<}Ind_{H}^{K}X,Y{>}_{K} = {<}X,Res^{K}_{H} Y{>}_{H}, 
\end{equation*}
and 
\begin{equation*}
{<}Res_{H}^{K}Y,X{>}_{H} = {<}Y,Ind^{K}_{H} X{>}_{K}, 
\end{equation*}
So we have a family of linear forms $(\phi_{H})_{H\leqslant G}$ on the Burnside algebras $(RB(H))_{H\leqslant G}$ defined by: let $X\in RB(H)$, then $\phi_{H}(X):={<}X,H/H{>}$. Let $H\leqslant K$ and $X\in RB(H)$, then 
\begin{align*}
\phi_{K}(Ind_{H}^{K}X)&={<}Ind_{H}^{K}(X),K/K{>}_{K}\\
&={<}X,Res^{K}_{H}K/K{>}_{H}\\
&={<}X,H/H{>}_{H}\\
&=\phi_{H}(X).
\end{align*}
The family $\big(\phi_{H}\big)_{H\leqslant G}$ is a stable by induction family of linear forms on the Burnside algebras $\big(RB(H)\big)_{H\leqslant G}$, and the bilinear forms $b_{\phi_{H}}$ are the bilinear forms ${<}-,-{>}_{H}$ so by definition they are non-degenerate. 
\end{proof}
\begin{re}
If the Mackey algebra is symmetric, it is always possible to choose a stable by induction family of linear maps $(\phi_{H})_{H\leqslant G}$ on $(RB(H))_{H\leqslant G}$ which generalize the trace maps on $\big(RH\big)_{H\leqslant G}$ in the sense of Remark \ref{gene}, i.e. such that $\phi_{H}(H/1)=1$.\newline Indeed, since the family is stable by induction, for every $H$ subgroup of $G$, we have $\phi_{H}(H/1)=\phi_{1}(1/1)$. Let us denote by $a$ the value $\phi_{H}(H/1)$. Now in the usual basis of $RB(H)$, the matrix of the bilinear form $b_{\phi_{H}}$ as a column divisible by $a$, and since this bilinear form is non degenerate, we have $a\in R^{\times}$, so one can normalize the linear forms $\phi_{H}$. 
\end{re}
\section{Symmetricity in the semi-simple case.}
Let $G$ be a finite group and $k$ a field of characteristic zero, or characteristic $p>0$ which does not divide the order of $G$, then it is well known that the Mackey algebra $\mu_{k}(G)$ is semi-simple, so it is clearly a symmetric algebra. One can specify a trace map for this algebra by using the previous section. Let us consider the linear form $\phi_{G}$ on $kB(G)$ defined by $$\phi(X) = \sum_{H \in [s(G)]} \frac{1}{|N_{G}(H)|} |X^{H}|,$$ where $X\in kB(G)$ and $[s(G)]$ is a system of representatives of conjugacy classes of subgroups of $G$. 
\newline In this situation the set of the primitive orthogonal idempotents of $kB(G)$ is well known. These idempotents are in bijection with the conjugacy classes of subgroups of $G$. If $H$ is a subgroup of $G$, let us denote by $e^{G}_{H}$ the idempotent corresponding to the conjugacy class of $H$. For more details, see \cite{yoshida_idempotent},\cite{gluck_idempotent} or \cite{bouc_burnside} for a summary. Let us recall some important results about these idempotents:
\begin{theo}
Let $G$ be a finite group.
\begin{enumerate}
\item Let $H$ and $K$ be subgroups of $G$, then $|(e_{H}^{G})^{K}| = 1$ if $H$ is conjugate to $K$ and $0$ otherwise.
\item Let $X$ be a $G$-set and $H\leqslant G$, then $X.e_{H}^{G} = |X^{H}|e^{G}_{H}$. 
\item Let $H\leqslant K$ be subgroups of $G$, then $Ind_{K}^{G}(e_{H}^{K})=\frac{|N_{G}(H)|}{|N_{K}(H)|}e_{H}^{G}.$
\item Let $H$ be a subgroup of $G$, then $$e_{H}^{G}=\frac{1}{|N_{G}(H)|} \sum_{K\leqslant H} |K|\mu(K,H) G/K.$$
\end{enumerate}
\end{theo}
\begin{lemme}\label{lee1}
\begin{enumerate}
\item Let $G$ be a finite group, then $\phi_{G}$ is a linear form.
\item The family $(\phi_{G})_{G}$ is stable by induction.
\item Let $G$ be a finite group, then $\phi_{G}(G/1)=1$. 
\end{enumerate}
\end{lemme}
\begin{proof}
The only non obvious assertion is the second. Since the map is linear it is enough to check this assertion on the basis elements of $kB(G)$. We use the basis consisting of the primitive orthogonal idempotents. Let $H\leqslant K\leqslant G$, then
\begin{align*}
\phi_{G}(Ind_{K}^{G}(e_{H}^{K}))&=\frac{|N_{G}(H)|}{|N_{K}(H)|}\phi_{G}(e^{G}_{H})\\
&=\frac{|N_{G}(H)|}{|N_{K}(H)|}\frac{1}{|N_{G}(H)|}\\
&=\frac{1}{|N_{K}(H)|}.
\end{align*}
In the other hand,
\begin{align*}
\phi_{K}(e_{H}^{K})=\frac{1}{|N_{K}(H)|}. 
\end{align*}
\end{proof}
\begin{prop}\label{prop1}
The determinant of this bilinear form $b_{\phi_{G}}$, in the basis consisting of the transitive $G$-sets is:
\begin{equation*}
det(b_{\phi})=\prod_{H\in [s(G)]} \frac{|N_{G}(H)|}{|H|^{2}}.
\end{equation*}
If $G$ is abelian, this determinant is equal to $1$.
\end{prop}
\begin{proof}
We first compute the determinant of this bilinear form in the basis consisting of the orthogonal primitive idempotents of $kB(G)$, then we apply a change of basis. Since the idempotents are orthogonal, this matrix is diagonal. The diagonal terms are $\phi_{G}(e_{H}^{G}) = \frac{1}{|N_{G}(H)|}$. So in this basis, the determinant of the matrix is $\prod_{H\in [s(G)]}\frac{1}{|N_{G}(H)|}$.
\newline The change of basis matrix from the basis of transitive $G$-sets to the basis of the primitive idempotents is a upper triangular matrix, the diagonal terms are the $\frac{|N_{G}(H)|}{|H|}$.  So in the basis of transitive $G$-sets, we have $$det(b_{\phi})=\prod_{H\in [s(G)]} \frac{|N_{G}(H)|}{|H|^{2}}.$$
\newline If $G$ is abelian, this determinant is equal to $\frac{\prod_{H\leqslant G}\frac{|G|}{|H|}}{\prod_{H\leqslant G}|H|}$, which is equal to $1$ since the abelian groups are isomorphic to their dual. 
\end{proof}
\begin{re}
There exist non abelian group such that this determinant is equal to $1$. The smallest counter example is for $G=(C_{4}\times C_{2})\rtimes C_{4}$. A quick run in GAP with the group $G:=SmallGroup(32,2)$ show that the determinant of $b_{\phi_{G}}$ is $1$.
\newline This determinant is most of the time of the form $\frac{1}{n}$, where $n\in \mathbb{N}$, but this is not always true. The first counter example is for two groups of order $64$: $H=(C_{8}\times C_{2})\rtimes C_{4}$ and $K=C_{2}\times ((C_{4}\times C_{2})\rtimes C_{4})$. The determinant is in these two cases $4$ and $16$.
\end{re}
\begin{coro}\label{cor}
Let $G$ be a finite group and $k$ be a field of characteristic zero, or $p>0$ which does not divide the order of $G$, then the Mackey algebra $\mu_{k}(G)$ is symmetric.
\end{coro}
\begin{proof}
By Lemma \ref{lee1} and Proposition \ref{prop1}, the family $(kB(H))_{H\leqslant G}$ is a stable by induction family of symmetric algebras. The result is now clear by Theorem \ref{meta}. 
\end{proof}
\section{Symmetry of the Mackey algebra over the ring of integers.}
The trace map defined in the previous section is not defined over the ring of integers. In this part let us consider the map $\phi_{G} : B(G) \to \mathbb{Z}$ defined on the usual basis by $\phi(G/H)= 1$ if $H=\{1\}$ and $\phi(G/H)=0$ otherwise. We have the following lemma:
\begin{lemme}
Let $G$ be a finite group.
\begin{enumerate}
\item $\phi_{G}$ is a linear form on $B(G)$.
\item $\phi=(\phi_{H})_{H\leqslant G}$ is a stable by induction family.
\item $\phi(G/1)=1$. 
\end{enumerate}
\end{lemme}
Let $G$ be a finite group. We denote by $\pi(G)$ the set of the prime divisors of $|G|$. Recall that for $\pi\subseteq \pi(G)$, a Hall-$\pi$-subgroup of $G$ (or a $S_{\pi}$-subgroup of $G$) is a $\pi$-subgroup $H$ such that $|H|$ and $|G/H|$ are coprime. The notion of $S_{\pi}$-group is a generalization of the notion of Sylow $p$-subgroup. In the case of a solvable group, there is a Sylow theorem for $S_{\pi}$-groups: 
\begin{theo}[Hall]
The group $G$ is solvable if and only if $G$ has $S_{\pi}$-subgroup for all set $\pi$ of prime divisors of $|G|$. In this case, 
\begin{enumerate} 
\item Two $S_{\pi}$-subgroups are conjugate in $G$.
\item Each $\pi$-subgroup of $G$ is contained in a $S_{\pi}$-subgroup.
\end{enumerate}
\end{theo}
\begin{proof}
The proof can be found in Part I.6 of \cite{gorenstein}.
\end{proof}
\begin{de}
The finite group $G$ is a square-free group if  $p^2$ does not divide the order of $G$ for any prime number $p$.  
\end{de}
Let us recall the well-known fact: 
\begin{lemme}
A square-free group $G$ is solvable.
\end{lemme}
\begin{proof}
The group $G$ is in fact a super-solvable group. This is well known, but we weren't able to find a reference. Let $p$ be the smallest prime divisor of $|G|$. Let $P$ be a Sylow $p$-subgroup of $G$. Then $N_{G}(P)/C_{G}(P)\hookrightarrow Aut(P)$. But $|Aut(P)|=p-1$ and the order of $N_{G}(P)/C_{G}(P)$ is a product of prime numbers bigger that $p$. So $N_{G}(P)=C_{G}(P)$, and by Burnside's Theorem, the set of all the $p'$-elements of $G$ is a normal subgroup of $G$. By induction this proves that $G$ is (super-)solvable. 
\end{proof}
\begin{coro}
Let $n$ be the size of $\pi(G)$. Then there are $2^n$ conjugacy classes of subgroups of $G$, one for each divisor of $|G|$. 
\end{coro}
\begin{proof}
Let $\pi$ be a set of prime divisors of $G$. Since $G$ is solvable, there is a $S_{\pi}$-subgroup of $G$. Now since $G$ is a square-free order group, each subgroup of $G$ is a $S_{\pi}$-subgroup for a set of prime $\pi$. So two subgroups are conjugate in $G$ if and only if they have the same order. 
\end{proof}
\begin{re}\label{ordre}
Let $\mathcal{P}$ be the set of divisors of $|G|$. Let us consider the following order on this set: let $p_1,p_2,\cdots, p_n$ be the prime divisors of $|G|$ such that $p_1 < p_2<\cdots <p_n$. Then $p_1<p_2<\cdots< p_n<p_1p_2<p_1p_3<\cdots <p_1p_n <p_2 p_3< \cdots <p_{n-1}p_{n} < p_1p_2p_3 <\cdots$. let $[H]$ and $[K]$ be two conjugacy classes of subgroups of $G$. Then $[H] \leqslant [K]$ if and only if $|H|<|K|$ for this order or $|H|=|K|$.  
\end{re}
\begin{prop}\label{det1}
Let $G$ be a square-free group. The determinant of the bilinear form $b_{\phi_{G}}$ is $\pm 1$.
\end{prop}
\begin{proof}
We will work with the basis of $B(G)$ consisting of transitive $G$-sets. Let $H$ and $K$ be subgroups of $G$, then
\begin{equation*}
b_{\phi}(G/H,G/K) = Card(\{g\in [H\backslash G/ K]\ ; \ H\cap K^{g} = 1\}).
\end{equation*}
\begin{itemize}
\item If $\pi(H)\sqcup \pi(K) = \pi(G)$ and $\pi(H)\cap \pi(K)= \emptyset$, then $b_{\phi}(G/H,G/K)=1$. Indeed, by cardinality reason, for all $g\in G$, we have $H\cap K^{g}=1$, so
\begin{equation*}
b_{\phi}(G/H,G/K)=Card\{g\in [H\backslash G/K]\}=1,
\end{equation*}
since there is only one double coset in this situation.
\item If $H \leqslant  G$ and $K\leqslant G$ such that $$\Pi_{p_i\in \pi(H)}p_i \times \Pi_{p_{j}\in\pi(K)}p_j>|G|,$$ then $b_{\phi}(G/H,G/K)=0$, since $H\cap K^{g}\neq \{1\}$ for all $g\in G$.
\end{itemize}
We order the basis elements using the total order of Remark \ref{ordre} on the subgroups of $G$. The antidiagonal coefficients of the matrix correspond to subgroups $H$ and $K$ such that $\pi(H)\cap \pi(K) =\emptyset$ and $\pi(H)\sqcup \pi(K)= \pi(G)$. So the anti-diagonal coefficients of the matrix are $1$. 
\newline The coefficients under the anti-diagonal correspond to subgroups  $H$ and $K$ such that  $\Pi_{p_i\in \pi(H)}p_i \times \Pi_{p_{j}\in\pi(K)}p_j>|G|$. So these coefficients are zero. The matrix of $b_{\phi}$ in this basis, is an upper anti-triangular matrix with $1$ on the anti-diagonal so its determinant is $\pm 1$. 
\end{proof}
\begin{theo}\label{sym}
The Mackey algebra $\mu_{\mathbb{Z}}(G)$ is a symmetric algebra if and only if $G$ is a square-free group.
\end{theo}
\begin{proof}
Let $G$ be a square-free group. Then by Theorem \ref{meta} and the result of Propostion \ref{det1}, the determinant of matrix of the bilinear form $(-,-)_{\phi} : \mu_{\mathbb{Z}}(G)\times \mu_{\mathbb{Z}}(G)\to \mathbb{Z}$ is $\pm 1$. There exist a non degenerate bilinear associative symmetric form for $\mu_{\mathbb{Z}}(G)$, so this algebra is symmetric. 
\newline Conversely, let $G$ be a finite group and $p$ be a prime number such that $p^2 \mid |G|$, then $G$ has a $p$-subgroup $P$ of order $p^2$. We prove that all the associative symmetric bilinear form ${<}-,-{>}$ on $RB(P)$ are degenerate. 
\begin{itemize}
\item Suppose that $P = C_{p^2}$, let $B$ be the Burnside functors of $Mack_{\mathbb{Z}}(P)$, then there are $a,b,c\in \mathbb{Z}$ such that the matrix $M$ of ${<}-,-{>}$ in the usual basis of $B(G)$ is:
\begin{equation*}
M=\left(\begin{array}{ccc}a & b & c \\b & pb & pc \\c & pc & p^2c\end{array}\right),
\end{equation*}
If we reduce modulo $p$ this matrix, it is clear that the two last columns are proportional. So the $det(M)$ is divisible by $p$, so $B$ is not isomorphic to its $\mathbb{Z}$-linear dual $B^{*}$.
\item Suppose that $P=C_{p} \times C_{p}$. Let $B$ be the Burnside functors of $Mack_{\mathbb{Z}}(P)$. There are elements $a,b_1,b_2,\cdots,b_{p+1},c\in \mathbb{Z}$ such that the  matrix $M$ of ${<}-,-{>}$ in the usual basis of $B(G)$ is:
\begin{equation*}
M:=\left(\begin{array}{cccccc}a & b_1 & \cdots & b_{p} & b_{p+1} & c \\b_{1} & pb_{1} & c & \cdots & c & pc \\\vdots & c & pb_{2} & \ddots & \vdots &  \\b_{p} & \vdots &  \ddots & \ddots & c & \vdots \\b_{p+1} & c & \cdots &  c & pb_{p+1} & pc \\c & pc & \cdots &  pc & pc & p^2c\end{array}\right)
\end{equation*}
By reducing this matrix modulo $p$ it is enough to look at the following $(p+1)\times (p+1)$ matrix:
\begin{equation*}
\left(\begin{array}{cccc}0 & c & \cdots & c \\c & 0 & \ddots & \vdots \\\vdots & \ddots & \ddots & c \\c & \cdots & c & 0\end{array}\right)
\end{equation*}
the sum of the lines is zero modulo $p$, so $det(M)$ is divisible by $p$. 
\end{itemize}
\end{proof}
\begin{re}
Let $G$ be a finite group and let $p$ be a prime number such that $p^2 \mid |G|$. The proof of Theorem \ref{sym} shows that if $p$ is not invertible in a commutative ring $R$, then the Mackey algebra $\mu_{R}(G)$ is not symmetric. 
\end{re}
\section{The p-local case.}
 Let $G$ be a finite group. Let $p$ be a prime number such that $p \mid |G|$. Let $R$ be a commutative ring with unit in which all the prime divisors of $|G|$ except $p$ are invertible. The ring $R$ can be  a field $k$ of characteristic $p>0$. If $(K,\mathcal{O},k)$ is a $p$-modular system, the ring $R$ can be either the valuation ring or the residue field. Finally $R$ can be the localization of $\mathbb{Z}$ at the prime $p$. 
\newline Even for the field $k$, the symmetry of the Mackey algebra  does not directly follows from Theorem \ref{sym}, since the determinant of the bilinear forms $(-,-)_{\phi}$ and $b_{\phi}$ can be zero. For example, the matrix of $b_{\phi_{C_{p^2}}}$ is: $\left(\begin{array}{ccc}p^2 & p & 1 \\p & 0 & 0 \\1 & 0 & 0\end{array}\right)$. So in characteristic $3$, for $G=C_{3}\times C_{4}$ the determinant of $b_{\phi_{G}}$ is zero. 
\newline Using Theorem \ref{main}, the symmetry of the Mackey algebra $\mu_{k}(G)$ follows from the symmetry of the modular Burnside algebras $(kB(H))_{H\leqslant G}$. In \cite{deiml}, Markus Deiml proved that the Burnside algebra of a finite group $G$ is symmetric if and only if $p^2\nmid |G|$. For our purpose, we need to check that the stability by induction condition holds. So, following Deiml's proof, we specify a symmetric associative non degenerate bilinear form on the Burnside algebra, then we check the stability condition. Almost all the arguments of Deiml can be used for the ring $R$, if it is not the case, we sketch the proof. 
\newline Let us recall that the primitive idempotents of the Burnside algebra are in bijection with the conjugacy classes of $p$-perfect subgroups of $G$ denoted by $[s(G)]_{perf}$. If $J$ is $p$-perfect, then we denote by $f_{J}^{G}$ the corresponding idempotent of $RB(G)$. 
\begin{lemme}[Lemma 3.4 \cite{yoshida_idempotent}.]\label{form}
Let $J$ be a $p$-perfect subgroup of $G$. Then, $f_{J}^{G}=\sum_{K} e_{K}^{G},$ where $K$ runs through the conjugacy classes of subgroups of $G$ such that $O^{p}(K)=_{G}J$. 
\end{lemme}
\begin{lemme}
Let $G$ be a finite group and $J$ be a $p$-perfect subgroup of $G$. If $p\mid \frac{|N_{G}(J)|}{|J|}$ and $p^{2}\nmid \frac{|N_{G}(J)|}{|J|}$, then there are exactly two conjugacy classes of subgroups $L$ of $G$ such that $O^{p}(L)=J$. 
\end{lemme}
\begin{proof}
Let $S_{J}\leqslant N_{G}(J)$ such that $S_{J}/J$ is a Sylow $p$-subgroup of $N_{G}(J)/J$, then $O^{p}(S_{J})=J$. Conversely if $H$ is a subgroup of $G$ such $O^{p}(H)$ is conjugate to $J$, then changing $H$ by one of its conjugate one can assume that $O^{p}(H)=J$ and $H\leqslant N_{G}(J)$. Now $H/J$ is a $p$-subgroup of $N_{G}(J)/J$, so there are two possibilities: either $H=J$ or $H/J$ is a Sylow $p$-subgroup of $N_{G}(J)/J$, i-e $H$ is conjugate to $S_{J}$.  
\end{proof}
\begin{lemme}[Lemma 5 \cite{deiml}]\label{basis1}
Let $J$ be a $p$-perfect group. Let us denote by $\mathcal{S}_{J}$ a set of representatives of conjugacy classes of subgroups $L$ of $G$ such that $O^{p}(L)=J$. Then the set of $G/I f_{J}^{G}$ where $I\in \mathcal{S}_{J}$ and $J\in [s(G)]_{perf}$ is a basis of $RB(G)$. 
\end{lemme}
\begin{proof}
Here, the proof of Deiml does not work for a general ring $R$, since there is a dimension argument. However by Lemma 5 (\cite{deiml}), we know that the family $\big(G/I f_{J}^{G}\big)$ is a free family, so we just need to check that it is a generating family. Let $K$ be a subgroup of $G$. It is enough to check that $G/K$ is a $R$-linear combination of elements of the form $G/I f_{J}^{G}$ where $O^{p}(I)=J$. If $|K|=1$, then $G/1 = G/1 f_{1}^{G}$. By induction on $|K|$, in $RB(G)$, we have:
\begin{align*}
G/K = G/K \times 1 &= \sum_{J\in [s(G)]_{perf}} G/K \times f_{J}^{G} \\
&= G/K f_{O^{p}(K)}^{G} + \sum_{O^{p}(K)\neq J\in [s(G)]_{perf}} G/K\times f_{J}^{G}.
\end{align*}
Now $G/K f_{J}^{G}$ is zero unless $J$ is conjugate to a subgroup of $K$. If it is the case, we have:
\begin{align*}
G/K f_{J}^{G} = \sum_{L\in [s(G)]\ ;\ O^{p}(L)=_{G}J} |G/K^{L}|e_{L}^{G}.
\end{align*}
Now $|G/K^{L}|$ is zero unless $L$ is conjugate to a subgroup of $K$. Moreover, since $O^{p}(L) = K \neq O^{p}(K)$, the group $L$ canot be equal to $K$. So $G/K f_{J}^{G}$ is a $R$-linear combination of transitive $G$-set $G/L'$ where $|L'|<|K|$. By induction, $G/K$ is a $R$-linear combination of elements of the form $G/I f_{J}^{G}$.
\end{proof}
Following \cite{deiml}, let us consider the linear form $\phi_{G}$ on $RB(G)$ defined on a basis element by:
\begin{equation*}
\phi\big(G/I f_{J}^{G}\big) = \left\{\begin{array}{c}1 \hbox{ if $I=J$,} \\0 \hbox{ if $I\neq J$.}\end{array}\right.
\end{equation*}
\begin{res}\label{re1}
\begin{itemize}
\item If $R=k$ is a field of characteristic $p$, and if $p\nmid |G|$, then the idempotents $f_{J}^{G}$ are the idempotents $e_{J}^{G}$ so it is easy to check that \begin{center}$\phi_{G}(X) = \sum_{H\in [s(G)]}\frac{|H|}{|N_{G}(H)|} |X^{H}|$, for $X\in kB(G)$. \end{center}
\item If $p\mid |G|$, it seems rather difficult to compute the value of $\phi_{G}$ on a transitive $G$-set. 
\end{itemize}
\end{res}
\begin{lemme}\label{indu}
Let $H\leqslant G$ and $J$ be a $p$-perfect subgroup of $H$. Then: 
\begin{enumerate}
\item $Ind_{H}^{G}(H/J f_{J}^{H})=G/Jf_{J}^{G}$.
\item Moreover if $p\mid |N_{H}(J)/J|$ and $p^2\nmid |N_{H}(J)/J|$, let $S_{J}$ be a subgroup of $H$ such that $J\subset S_{J}$ and $O^{p}(S_{J})=J$. Then:
$$Ind_{H}^{G}(H/S_{J}f_{J}^{H})=G/S_{J} f_{J}^{G}.$$
\end{enumerate}
\end{lemme}
\begin{proof}
Using Lemma \ref{form}, we have: 
\begin{align*}
Ind_{H}^{G}\big(H/J f_{J}^{H}\big)&= \frac{|N_{H}(J)|}{|J|}Ind_{H}^{G}(e_{H}^{G})\\
&=\frac{|N_{H}(J)|}{|J|} \frac{|N_{G}(J)|}{|N_{H}(J)|} e_{J}^{G}\\
&=G/J f_{J}^{G}.
\end{align*}
For the second part, by Lemma $3.5$ of \cite{yoshida_idempotent}, we have $Res^{G}_{H}(f^{G}_{J})=\sum_{J'} f^{H}_{J'}$ where $J'$ runs the subgroups of $H$ up to $H$-conjugacy such that $J'$ is conjugate to $J$ in $G$. So, we have:
\begin{align*}
H/S_{J} Res^{G}_{H}(f^{G}_{J})&=\sum_{J'} H/S_{J} f^{H}_{J'},
\end{align*}
but we have: 
\begin{align*}
H/S_{J}f_{J'}^{H}= \sum_{\underset{O^{p}(K)=J'}{K\leqslant J\hbox{ {\footnotesize up to $H$-conjugacy}}}} |(H/S_{J})^K|e_{K}^{H}. 
\end{align*}
But $|(H/S_{J})^{K}|=0$ unless $K$ is $H$-conjugate to a subgroup of $S_{J}$. Without lost of generality one can assume $K\subseteq S_{J}$. So the only non zero terms are for $J'\leqslant K \leqslant S_{J}$ and since $|S_{J}|/|J'|=p$ either $K=J'$ or $K=S_{J}$. If $K=S_{J}$, then $O^{p}(K)=J'$ is $H$-conjugate to $O^{p}(S_{J})=J$, that is $J'$ is $H$-conjugate to $J$.  
\newline If $K=J'$ and $J\neq J'$, then there we have the following situation:
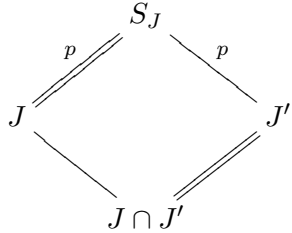
\begin{equation*}
\xymatrix{
& S_{J} &\\
J\ar@{=}[ur]^{p} & & J'\ar@{-}[ul]_{p}\\
& J\cap J'\ar@{-}[ul]\ar@{=}[ur] &
}
\end{equation*}
The two subgroups $J$ and $J'$ are of index $p$ in $S_{J}$. We have $JJ'=S_{J}$. Since $J$ is normal in $S_{J}$, the intersection $J\cap J'$ is normal in $J'$. Then by the second isomorphism theorem, we have $|J'|/|J'\cap J|=p$. This implies that $p^2/|S_{J}|$ which is not possible by hypothesis. So we have $H/S_{J}Res^{G}_{H}(f_{J}^{G})=H/S_{J} f_{J}^{H}$. Using the Frobenius identity (see Proposition $3.13$ \cite{bouc_burnside}), we have:
\begin{align*}
Ind_{H}^{G}(H/S_{j} f_{J}^{H})&=Ind_{H}^{G}(H/S_{J}Res^{G}_{H}f^{G}_{J})\\
&= G/S_{J}f^{G}_{J}. 
\end{align*}
\newline 
\end{proof}
\begin{lemme}\label{mod1}
Let $G$ be a finite group.
\begin{enumerate}
\item $\phi_{G}(G/1)=1$.
\item if $p\mid |G|$ and $p^2 \nmid |G|$, then the family $\big(\phi_{H}\big)_{H\leqslant G}$ is stable by induction. 
\end{enumerate}
\end{lemme}
\begin{proof}
\begin{enumerate}
\item The first part is obvious since $G/1f_{1}^{G}=|G|e^{G}_{1}=G/1$.
\item The second part follow from Lemma \ref{indu}. 
\end{enumerate}
\end{proof}
\begin{prop}\label{mod2}
Let $G$ be a finite group such that $p\mid |G|$ and $p^2\nmid |G|$. Then the Burnside algebra  $RB(G)$ is a symmetric algebra. \end{prop}
\begin{proof}
In the basis of Lemma \ref{basis1} the matrix of $b_{\phi_{G}}$ is a diagonal by block matrix. The blocks are indexed by the conjugacy classes of $p$-perfect subgroups of $G$. If $J$ is a $p$-perfect subgroup such that $p\nmid |N_{G}(J)/J|$, then there is only one conjugacy class of subgroup $L$ of $G$ such that $O^{p}(L)=J$, so the block indexed by $J$ is of size $1$. The entry in this block is :
\begin{align*}
b_{\phi_{G}}(G/Jf_{J}^{G},G/Jf_{J}^{G}) &= \phi(G_{J}f_{J}^{G}\times G/Jf_{J}^{G})\\
&= \sum_{g\in [J\backslash G/J]} \phi(G/J\cap J^{g}f_{J}^{G})\\
&=\sum_{g\in [N_{G}(J)/J]}\phi(G/Jf_{J}^{G})\\
&=\frac{|N_{G}(J)|}{|J|} \in R^{\times}.  
\end{align*}
If $J$ is a $p$-perfect subgroup of $G$ such that $p\mid |N_{G}(J)/J|$, then there are two conjugacy classes of subgroups $L$ of $G$ such that $O^{p}(L)=J$. We denote by $S_{J}$ a subgroup of $G$ such that $J\subset S_{J}$ and $O^{p}(S_{J})=J$. The block matrix indexed by $J$ is of size $2$. The first diagonal entry is:
\begin{align*}
b_{\phi_{G}}(G/Jf_{J}^{G},G/Jf_{J}^{G})=\frac{|N_{G}(J)|}{|J|}. 
\end{align*}
the anti-diagonal entries are:
\begin{align*}
b_{\phi_{G}}(G/J\times G/S_{J} f_{J}^{G})&=\sum_{g\in [S_{J}\backslash G/J]} \phi(G/S_{J}\cap J^{g} f_{J}^{G})\\
&=\sum_{g\in [S_{J}\backslash N_{G}(J)/J]} 1\\
&=\frac{|N_{G}(J)|}{|S_{J}|}.
\end{align*}
Finally, the second diagonal element is:
\begin{align*}
a:=b_{\phi_{G}}(G/S_{J}f_{J}^{G},G/S_{J}f_{J}^{G})= \sum_{g\in [S_{J}\backslash G/S_J]} \phi_{G}(G/S_{J}\cap S_{J}^{g} f_{J}^{G}).
\end{align*}
Now, if $g\notin N_{G}(J)$ we have $G/S_{J}\cap S_{J}^{g} f_{J}^{G}=0$ and if $g\in N_{G}(S_{J})$, we have $$\phi_{G}(G/S_{J}f_{J}^{G})=0.$$ For the computation of $a$, we work in $\mathbb{Q}$. Then we have:
\begin{align*}
a&=\sum_{g\in [S_{J}\backslash G/S_J]} \phi_{G}(G/S_{J}\cap S_{J}^{g} f_{J}^{G})\\
&= \sum_{g\in N_{G}(J)\backslash N_{G}(S_{J})} \frac{|S_{J}\cap S_{J}^{g}|}{|S_{J}|^{2}} \phi_{G}(G/J f_{J}^{G})\\
&=\sum_{g\in N_{G}(J)\backslash N_{G}(S_{J})} \frac{|J|}{|S_{J}|^2}\\
&= \frac{|J|}{|S_{J}|^2}\big(|N_{G}(J)|-|N_{G}(S_{J})|).
\end{align*}
The determinant of each of these blocks is:
\begin{align*}
&\frac{|N_{G}(J)|}{|J|}\times \Big(\frac{|J|}{|S_{J}|^2}\big(|N_{G}(J)|-|N_{G}(S_{J})|)\Big)- \frac{|N_{G}(J)|^2}{|S_{J}^2|}\\
& = -\frac{|N_{G}(J)|\times |N_{G}(S_{J})|}{S_{J}^{2}}\in R^{\times}.
\end{align*}
This determinant is invertible in $R$, so the bilinear form $b_{\phi_{G}}$ is non degenerate. 
\end{proof}
\begin{coro}
Let $G$ be a finite group. Then the Mackey algebra $\mu_{R}(G)$ is a symmetric algebra if and only if $p^2\nmid |G|$.
\end{coro}
\begin{proof}
If $p^2\nmid |G|$, the fact that $\mu_{R}(G)$ is a symmetric algebra follows from Theorem \ref{meta}, Proposition \ref{mod2} and Lemma \ref{mod1}. If $p^2\mid |G|$, we saw in the proof of Theorem \ref{sym} that every associative bilinear form on $RB(P)$ is degenerate if $|P|=p^2$, so the Mackey algebra $\mu_{R}(G)$ is not a symmetric algebra. 
\end{proof}
\paragraph{Acknowledgements}
The author would like to thank the foundation FEDER and the CNRS for their financial support and the foundation ECOS and CONACYT for the financial support in the project M10M01. Thanks also go to Serge Bouc for his suggestions. 
\bibliographystyle{abbrv}

\vspace{5pt}
\par\noindent
{Baptiste Rognerud\\
EPFL / SB / MATHGEOM / CTG\\
Station 8\\
CH-1015 Lausanne\\
Switzerland\\
e-mail: baptiste.rognerud@epfl.ch}
\end{document}